\documentclass[12pt]{amsart}

\usepackage{rotating}
\usepackage{pdfpages}
\usepackage{mathdots}
\usepackage{comment}
\usepackage{graphicx,eepic}
\usepackage{amsmath,amsfonts,amssymb, amsthm}
\usepackage{a4wide}

\makeatletter
\@namedef{subjclassname@2010}{%
  \textup{2010} Mathematics Subject Classification}
\makeatother

\newtheorem{thm}{Theorem}[section]
\newtheorem{cor}[thm]{Corollary}
\newtheorem{lemma}[thm]{Lemma}
\newtheorem{prop}[thm]{Proposition}

\theoremstyle{definition}
\newtheorem{rmk}[thm]{Remark}

\numberwithin{equation}{section}

\newcommand{\BbR}{\mathbb{R}}
\newcommand{\BbQ}{\mathbb{Q}}
\newcommand{\BbC}{\mathbb{C}}
\newcommand{\BbN}{\mathbb{N}}
\newcommand{\la}{\lambda}
\newcommand{\ka}{\kappa}

\newcommand{\be}{\beta}

\newcommand{\U}{{\mathcal{U}}}

\newcommand{\wpi}{\widetilde\pi}
\newcommand{\wta}{\widetilde A}

\renewcommand{\emptyset}{\varnothing}

\begin{document}

\baselineskip=17pt

\title[Multidimensional self-affine sets]
{Multidimensional self-affine sets: non-empty interior and the set of uniqueness}

\author{Kevin G. Hare}
\address{Department of Pure Mathematics \\
University of Waterloo \\
Waterloo, Ontario \\
Canada N2L 3G1}
\email{kghare@uwaterloo.ca}
\thanks{Research of K. G. Hare was supported by NSERC Grant RGPIN-2014-03154}

\author{Nikita Sidorov}
\address{School of Mathematics \\
The University of Manchester \\
Oxford Road, Manchester M13 9PL\\
 United Kingdom.}
 \email{sidorov@manchester.ac.uk}
\date{}

\begin{abstract}
Let $M$ be a $d\times d$ real contracting matrix. In this paper
we consider the self-affine iterated function system $\{Mv-u, Mv+u\}$, where $u$ is a cyclic vector.
Our main result is as follows: if $|\det M|\ge 2^{-1/d}$, then the attractor $A_M$
has non-empty interior.

We also consider the set $\mathcal U_M$ of points
in $A_M$ which have a unique address. We show that unless $M$
belongs to a very special (non-generic) class, the Hausdorff dimension
of $\mathcal U_M$ is positive. For this special class the full description
of $\mathcal U_M$ is given as well.

This paper continues our work begun in \cite{HS, HS2D}.
\end{abstract}



\maketitle

\section{Non-empty interior}

Let $d\ge2$ and $M$ be a $d\times d$ real matrix whose eigenvalues are all less than~1 in modulus.
Denote by $A_M$ the attractor for the contracting self-affine iterated function system (IFS) $\{Mv-u, Mv+u\}$,
i.e., $A_M=\{\pi_M(a_0a_1\dots) \mid a_n\in\{\pm1\}\}$, where
\[
\pi_M(a_0a_1\dots)=\sum_{k=0}^\infty a_k M^k u.
\]
If $A_M\ni x=\pi_M(a_0a_1\dots)$, then we call the sequence
    $a_0a_1\dots\in\{\pm1\}^{\mathbb N}$ an {\em address} of $x$. We assume our IFS to be
    {\em non-degenerate}, i.e., $A_M$ does not lie in any $(d-1)$-dimensional subspace of $\mathbb R^d$
    (i.e., $A_M$ {\em spans} $\mathbb R^d$).
    Let $u\in\mathbb R^d$ be a {\em cyclic vector} for $M$,
i.e., $\text{span}\{M^n u \mid n\ge0\}=\mathbb R^d$.

Our main result is as follows.

\begin{thm}\label{thm:nei}
If
\[
|\det M|\ge 2^{-1/d},
\]
then the attractor $A_M$ has non-empty interior. In particular, this is the case
when each eigenvalue of $M$ is greater than $2^{-1/d^2}$ in modulus.
\end{thm}

\begin{rmk}Note that if $|\det M|<\frac12$, then $A_M$ is a null set (see \cite{Gunturk})
and therefore, has empty interior. It is an interesting question whether $2^{-1/d}$
in Theorem~\ref{thm:nei} can be replaced with a constant independent of $d$.
\end{rmk}

\begin{cor}For an IFS $\{Mv+u_j\}_{j=1}^m$ with $m\ge2$ the same claim holds,
provided the IFS is non-degenerate.
\end{cor}
\begin{proof}Clearly, the attractors are nested as $m$ increases, so it suffices
to establish the claim for $m=2$. This, in turn, follows from Theorem~\ref{thm:nei}
via an affine change of coordinates.
\end{proof}

The history of the problem is as follows. In \cite{DJK} it was shown that for
$M=\begin{pmatrix} \la & 0 \\ 0 & \mu \end{pmatrix}$, if $0.953<\la<\mu<1$, then
$(0,0)$ has a neighbourhood which lies in $A_M$. Their method was a modification
of the one suggested in \cite{Gunturk}. In \cite{HS} we improved their lower
bound to $0.83$. In \cite{HS2D} we proved analogous results for all $2\times2$ matrices
$M$ by using a similar approach as in \cite{HS} for the matrices with real eigenvalues and
a different one for the rest. This second approach is the one we use
in the current paper.

To prove Theorem~\ref{thm:nei}, we need some auxiliary results. These are natural generalizations of those
from \cite[Appendix]{HS2D} whose proofs had been provided by V.~Kleptsyn \cite{MO}.
We use $+$ for the Minkowski sum of two sets:
\[
A+B=\{a+b\mid a\in A,\ b\in B\}.
\]

\begin{lemma}
\label{lem:Victor}
Let $\gamma$ be a path in $\BbR^n$.
Let $\gamma'(t_1, t_2, \cdots, t_{n-1}) =
    \gamma_1(t_1) + \dots + \gamma_{n-1}(t_{n-1})$ where the $\gamma_i$
    are paths in $\BbR^n$.
Let $\delta$ be the diameter of $\gamma([s_1,s_2])$, and assume that
    there is no point in the interior of the surface
$\sigma=\{\gamma(s) + \gamma'(t) :s,t\in\partial([s_1, s_2] \times [0,1]^{n-1})\}$.
Then the sets $\gamma(s_1) + \gamma'([0,1]^{n-1})$ and $\gamma(s_2) + \gamma'([0,1]^{n-1})$
    coincide outside a $\delta$-neighbourhood of
    $\gamma([s_1, s_2])+\gamma'(\partial([0,1]^{n-1}))$.
\end{lemma}

\begin{proof}
Assume the contrary and let $z$ be a point of the surface
    $\widetilde\gamma:=\gamma(s_1)+\gamma(t_1)$ (for some
    $t_1 \in [0,1]^{n-1}$) that lies outside the $\delta$-neighbourhood and that does
    not belong to the surface $\gamma(s_2)+\gamma([0,1]^{n-1})$.
By continuity, there is a $\varepsilon$-neighbourhood of $z$ that the latter surface does
    not intersect.

Now, by the Jordan-Brouwer separation theorem, in this neighbourhood one can find two
    points ``on different sides'' with respect to $\widetilde\gamma$.

This implies that one of these two points is in the interior of
    $\sigma=\{\gamma(s) + \gamma'(t) :s,t\in\partial([s_1, s_2] \times [0,1]^{n-1})\}$.
\end{proof}

\begin{prop}
\label{prop:Victor}
If $\gamma_1, \gamma_2, \cdots, \gamma_n$ be $n$ paths in $\BbR^n$ whose span is $\BbR^n$,
    then $\gamma_1 + \gamma_2 + \cdots + \gamma_n$ has non-empty interior.
\end{prop}

\begin{proof}
Let $t = (t_2, t_3, \dots, t_n), \gamma(t_1) = \gamma_1(t_1)$ and
    $\widehat\gamma(t) = (\gamma_2(t_2), \gamma_3(t_3), \dots, \gamma_n(t_n))$.
Consider the surface
\[
\omega:=\{\gamma(s) + \widehat\gamma(t) : (s,t)\in\partial([0,1] \times [0,1]^{n-1})\}.
\]

Let $\delta=\delta(s_1,s_2)$ be the diameter of $\gamma([s_1, s_2])$ for $s_1, s_2 \in [0,1]$.
Clearly, $\delta \to 0$ as $s_1 \to s_2$.
Pick $s_1$ and $s_2$ sufficiently close so the diameter of $\gamma_i([0,1])$
    is greater than $2\delta$ for all $i$.
Hence there exists a point on the surface $\gamma(s_1) + \widehat\gamma([0,1]^{n-1})$ that is
    not in the $\delta$-neighbourhood of $\gamma([s_1, s_2]) + \gamma(\partial([0,1]^{n-1}))$.
By Lemma~\ref{lem:Victor}, either there exists a point in the interior of this surface,
    or $\gamma(s_1) + \widehat\gamma([0,1]^{n-1})$ and $\gamma(s_2) + \widehat\gamma([0,1]^{n-1})$
    coincide outside the $\delta$-neighbourhood of
    $\gamma([s_1, s_2]) + \widehat\gamma(\partial([0,1]^{n-1}))$.

Taking $s_1 \to s_2$ and assuming that there is never a point in the interior gives
    that $\widehat\gamma([0,1]^{n-1})$ admits an arbitrarily small translation symmetry outside
    its endpoints. This in turn gives that $\widehat\gamma([0,1]^{n-1})$ is a $n-1$ dimensional plane, and that
    $\gamma([0, 1])$ lies within this plane. Hence $\gamma_1, \gamma_2, \dots, \gamma_n$
    do not span is $\BbR^n$, a contradiction.
\end{proof}

We need two more results before we can get on with the proof of Theorem~\ref{thm:nei}.

\begin{lemma}\cite[Lemma~2.3]{ShSo} \label{lem:shso}
The set $A_M$ is connected if $|\det M|\ge\frac12$.
\end{lemma}

\begin{lemma}\label{lem:odl}\cite[Lemma~4.1]{OP}
Let $Y$ be a topological space.
Suppose $f:\{m,p\}^\BbN \to Y$ is a continuous map such that
    \[ f([wm]) \cap f([wp]) \neq \emptyset \]
for all $w \in \{m, p\}^*$. (Here $m$ stands for $-1$ and $p$ for $1$.)
Then the image of $f$ is path connected.
\end{lemma}
Here $[i_1\dots i_k]$ is the cylinder $\{\{a_j\}_{j=1}^\infty \subset \{p, m\}^\mathbb N
\mid a_j = i_j,\ j = 1, \dots, k\}$.

Using $f := \pi_M$ and $Y = \BbR^d$, we see that $A_M$ is the image of $f$.
This gives the following corollary.
\begin{cor}\label{cor:path-connected}
The set $A_M$ is path connected if $|\det M|\ge\frac12$.
\end{cor}

\begin{proof}[Proof of Theorem~\ref{thm:nei}]
Let us first change the set of ``digits'' for this particular proof.
Namely, consider the affine change of coordinates $x\mapsto \frac12\left(x+\sum_{k=0}^\infty M^k u\right)$; this
change corresponds to $a_k\mapsto \frac12(a_k+1)\in\{0,1\}$. Recall that $u$ is
chosen to be a cyclic vector.  Thus, we have
\begin{align*}
\wta_M & = \{\wpi_M(a_0a_1\dots) \mid a_k\in\{0,1\}\}\\
       &=\left\{\sum_{k=0}^\infty a_kM^ku \mid a_k\in\{0,1\}\right\}\\
&=\left\{\sum_{n=0}^\infty\sum_{j=0}^{d-1} a_{dn+j}M^{dn+j}u\mid a_{dn+j}\in\{0,1\}\right\}\\
&=\left\{\sum_{j=0}^{d-1} M^j\sum_{n=0}^\infty a_{dn+j}(M^d)^nu\mid a_{dn+j}\in\{0,1\}\right\}\\
&=\wta_{M^d}+M\cdot \wta_{M^d}+\dots + M^{d-1}\cdot\wta_{M^d},
\end{align*}
where $M\cdot X=\{Mx : x\in X\}$.
Now, if $|\det M^d|\ge\frac12$, then by Corollary~\ref{cor:path-connected},
the attractor $\wta_{M^d}$ is path connected. We have $u=\wpi_{M^d}(1000\dots)$, whence
$u\in\wta_{M^d}$. Notice that
\begin{equation}\label{eq:span}
\text{span}\{M^n u \mid n\ge0\}=\text{span}\{M^n u \mid 0\le n\le d-1\}=\mathbb R^d,
\end{equation}
since $M^n$ is a linear combination of $I,M,\dots, M^{d-1}$ for all $n\ge d$,
in view of the Cayley-Hamilton theorem.

Choose now any path $\gamma$ in $\wta_{M^d}$ which contains $u$. By (\ref{eq:span}),
the paths $\gamma, M\gamma,\dots, M^{d-1}\gamma$ span $\mathbb R^d$
as well, whence by Proposition~\ref{prop:Victor}, $\wta_M$ has non-empty interior, and thus,
so does $A_M$.
\end{proof}

\begin{rmk}
For $d=2$, Theorem~\ref{thm:nei} implies that if both eigenvalues of $M$
are greater than or equal to $2^{-1/4}\approx 0.8409$ in modulus, then
$A_M$ has non-empty interior. This is essentially \cite[Theorem~1.1]{HS2D}. Notice, however, that for
$M$ having real eigenvalues in \cite[Theorem~1.1]{HS2D} contains better bounds, due to a different
proof. In particular, if $M=\begin{pmatrix} -\la & 0 \\ 0 & \mu \end{pmatrix}$ with $0<\la\le\mu<1$,
then we have the same claim with $\la\ge2^{-1/2}\approx 0.7071$, and this bound is sharp if
$\la=\mu$.
\end{rmk}

\section{The set of uniqueness}

Let $\mathcal U_M$ be the {\em set of uniqueness} for our IFS, i.e., the set of $x\in A_M$ each of which
    has a unique address.
We let $U_M$ denote the set of unique addresses for $A_M$, so $\mathcal U_M=\pi_M(U_M)$.
For $d=1$ the set of uniqueness is a well studied topic -- see, e.g., \cite{KdV} and references
therein.

When $d=2$, the following result holds:
\begin{thm}
\label{thm:old}
Let $M$ be a contractive $2 \times 2$ matrix which we assume to be -- after an appropriate
change of coordinates - one of the following:
\begin{enumerate}
\item $M = \begin{pmatrix}\la & 1 \\ 0 & \la \end{pmatrix}$.
    For any $\la \neq 0$, the set of
    uniqueness has positive Hausdorff dimension. \cite[Corollary 4.8]{HS2D}.
\item $M = \begin{pmatrix}\la_1 & 0 \\ 0 & \la_2 \end{pmatrix}$.
      For any $0 < \la_1 < \la_2 < 1$, the set of uniqueness has positive Hausdorff dimension.
    \cite[Corollary 4.3]{HS}.\label{case2}
\item $M = \begin{pmatrix}\la_1 & 0 \\ 0 & \la_2 \end{pmatrix}$.
    For any $-1 < \la_1 < 0 < \la_2 < 1$ with $|\la_1| \neq |\la_2|$, the set of
    uniqueness has positive Hausdorff dimension. \cite[Corollary 4.5]{HS2D}.\label{case3}
\item $M = \begin{pmatrix}a & b \\ -b & a \end{pmatrix}$ with $\ka = a+b i$.
      For any $\ka$ with $\arg(\ka)/\pi \not\in \BbQ$ the set of
      uniqueness has positive Hausdorff dimension. \cite[Section 4.3.1]{HS2D}.
\item $M = \begin{pmatrix}a & b \\ -b & a \end{pmatrix}$ with $\ka = a+b i$.
      For any $\ka$ with $\arg(\ka)/\pi \in \BbQ$ set $q >0$ minimal such that
          $\ka^q \in \BbR$ and let $\be = |\ka|^{-q}$.
      Then the set of uniqueness $\mathcal U_M$ is as follows:
      \begin{enumerate}
      \item finite non-empty if $\be\in (1,G]$;
      \item infinite countable for $\be\in(G,\be_*)$;
      \item an uncountable set of zero Hausdorff dimension if
      $\be=\be_*$; and
      \item a set of positive Hausdorff dimension for
      $\be\in (\be_*,\infty)$.
\end{enumerate}
      \cite[Theorem 4.16]{HS2D}.
\item \label{case:signs} $M=\begin{pmatrix} -\la & 0 \\ 0 & \la \end{pmatrix}$ with $0<\la<1$. Then we have
      the same claim as in the previous item with $\be=\la^{-2}$. \cite[Proposition~4.21]{HS2D}
\end{enumerate}
\end{thm}

Here $G=\frac{1+\sqrt5}2$ and
    $\be_*\approx 1.7872$ is the \textit{Komornik-Loreti constant} introduced
    in \cite{KL}.
The Komornik-Loreti constant is defined as the
unique solution of the equation $\sum_{n=1}^{\infty}\mathfrak{m}_{n}x^{-n+1}=1$,
where $\mathfrak{m}=(\mathfrak{m}_n)_1^\infty$ is the Thue-Morse
sequence
$$
\mathfrak{m}=0110\,\,1001\,\,1001\,\,0110\,\,1001\,\,0110\dots,
$$
i.e., the fixed point of the substitution $0\to01,\ 1\to10$.

The following result is straightforward.
\begin{lemma}\label{lem:block}
Let $M$ be a block matrix, i.e.,
\[
M=\begin{pmatrix} M_1 & 0 \\ 0 & M_2 \end{pmatrix}.
\]
Then $U_M\supset U_{M_j}$ for $j\in\{1,2\}$.
\end{lemma}

\begin{proof}
Notice that
\[
\pi_M(a_0 a_1 a_2 \dots) =
\begin{pmatrix} \pi_{M_1}(a_0a_1\dots) \\ \pi_{M_2}(a_0a_1\dots)
\end{pmatrix}.
\]
We see that if one of the two coordinates on the right hand side is
    unique, then the left hand side must also be unique.
\end{proof}

\begin{cor}
If $\dim_H \U_{M_1} > 0$ or
   $\dim_H \U_{M_2} > 0$, then
   $\dim_H \U_{M} > 0$.
\end{cor}

\begin{rmk}
Note that this claim is not if and only if.
To see this, take, for instance, $M_1$ and $M_2$ both $1 \times 1$ real matrices with positive
eigenvalues $\la\in\bigl(\frac{\sqrt5-1}2,1\bigr)$ and $\mu\in\bigl(\frac{\sqrt5-1}2,1\bigr)$
   with $\la \neq \mu$.
Then $\dim_H \U_{M}>0$ (\cite[Corollary~4.3]{HS}),
whereas $\mathcal U_{M_1}$ and $\mathcal U_{M_2}$ are finite -- see \cite{EJK}.
\end{rmk}

By converting a matrix $M$ to Jordan normal form, this gives a rich
    family of matrices for which $\dim_H \U_{M} > 0$.
In particular, this allows us to prove
\begin{thm}
\label{thm:new}
Let $M$ be a $d \times d$ matrix.
\begin{enumerate}
\item If $M$ has a non-trivial Jordan block, then $\dim_H \mathcal U_{M} > 0$.
\label{case:jordan}
\item If $M$ has an eigenvalue $\ka$ with $\arg(\ka)/\pi \not\in \BbQ$, then
      $\dim_H \mathcal U_{M} > 0$.
\label{case:irrational}
\item If $M$ has two eigenvalues $\ka_1$ and $\ka_2$ with $|\ka_1| \neq |\ka_2|$, then
      $\dim_H \mathcal U_{M} > 0$.
\label{case:different}
\item Let $M$ have only distinct simple eigenvalues, $\ka_1, \ka_2, \dots, \ka_d$
    with $\arg(\ka_j)/\pi \in \BbQ$ for all $j$.
    Assume further $|\ka_1| = \dots = |\ka_d|$.
    Let $q\in\mathbb N$ be minimal such that $\ka_j^q \in \BbR, 1\le j\le d$.
    If there exists $j$ and $k$ such that $\kappa_j^q \kappa_k^q < 0$, then
        put $\beta = |\ka_1|^{-2q}$, otherwise put $\beta = |\ka_1|^{-q}$.
    Then the set of uniqueness $\mathcal U_M$ is as follows:
\label{case:rational}
      \begin{enumerate}
      \item finite non-empty if $\be\in (1,G]$;
      \item infinite countable for $\be\in(G,\be_*)$;
      \item an uncountable set of zero Hausdorff dimension if
      $\be=\be_*$; and
      \item a set of positive Hausdorff dimension for
      $\be\in (\be_*,\infty)$.
\end{enumerate}
\end{enumerate}
\end{thm}

Some of these follow directly from Theorem~\ref{thm:old} and Lemma~\ref{lem:block}.
In Section~\ref{sec:jordan block} we show the case of Jordan blocks of size greater than
    or equal to $3$, and Jordan blocks of complex eigenvalues.
That is, we show Theorem~\ref{thm:new}~(\ref{case:jordan}).
Theorem~\ref{thm:new}~(\ref{case:irrational}) follows directly from Theorem~\ref{thm:old}
    and Lemma~\ref{lem:block}.
In Section~\ref{sec:simple block} we prove cases (\ref{case:different}) and
    (\ref{case:rational}).

There is a natural correspondence between a $2 \times 2$ real matrix
    $\begin{pmatrix} a & b \\ -b & a \end{pmatrix}$ and
     the $1 \times 1$ complex matrix $\begin{pmatrix} a+ bi \end{pmatrix}$.
For notational reasons, we will often use this second form for a matrix or sub-matrix
    corresponding to a complex eigenvalue of $M$.

\section{Jordan blocks}
\label{sec:jordan block}

\begin{lemma}\label{lem:jordan}
Let \[ M = \begin{pmatrix}
\kappa & 1       &        &         &  0 \\
        & \kappa & 1      &         &  \\
        &         & \ddots & \ddots  &  \\
        &         &        & \kappa & 1  \\
  0     &         &        &         & \kappa
\end{pmatrix} \]
with $0 < |\kappa| < 1$.
Then $\dim_H \U_M >0$.
\end{lemma}

\begin{proof}
First, assume that $\kappa \in \BbR$.
Let
\[
M' = \begin{pmatrix}
     \ka &1 \\
     0  &\ka
    \end{pmatrix}.
\]

By \cite[Lemma~3.1]{HS2D}, we have:
    \[ \pi_M(a_0 a_1 a_2 \dots) =
        \begin{pmatrix}
        \frac{1}{(k-1)!}\frac{d^{k-1}}{d \ka^{k-1}} \sum_{j=0}^\infty a_j \ka^j \\
        \frac{1}{(k-2)!}\frac{d^{k-2}}{d \ka^{k-2}} \sum_{j=0}^\infty a_j \ka^j \\
        \vdots \\
        \frac{d}{d \ka} \sum_{j=0}^\infty a_j \ka^j \\
        \sum_{j=0}^\infty a_j \ka^j
        \end{pmatrix}
    \]
and for $M'$ we have
    \[ \pi_{M'}(a_0 a_1 a_2 \dots) =
        \begin{pmatrix}
        \frac{d}{d \ka} \sum_{j=0}^\infty a_j \ka^j \\
        \sum_{j=0}^\infty a_j \ka^j
        \end{pmatrix}.
    \]
(Here we are assuming $u$ our cyclic vector is $\begin{pmatrix}0 & \dots & 0 & 1
    \end{pmatrix}^T$.)
Hence if $a_0 a_1\dots \in U_{M'}$, then
    the last two coordinates of $\pi_M(a_0 a_1 \dots)$ form a unique pair,
    whence $a_0 a_1 \dots \in U_M$. As $\dim_H\U_{M'} > 0$ from
    \cite[Corollary 4.8]{HS2D}, the result follows.

Next assume that $\ka \not \in \BbR$.
If $\arg(\kappa)/\pi \not\in \BbQ$, then we can repeat the above proof with
    $M' = \begin{pmatrix} \kappa \end{pmatrix}$ and \cite[Section 4.3.1]{HS2D}.
So assume that $\arg(\kappa/\pi) \in \BbQ$.
From the techniques above, we see that it suffices to show the
    $2 \times 2$ case, after which the result will follow.
Let \[ M = \begin{pmatrix}
\kappa & 1       \\
  0     & \kappa
\end{pmatrix} \]
with $0 < |\kappa| < 1$, $\arg(\kappa)/\pi \in \BbQ$.
Let $(z_1,z_2) \in A_M$ with $z_1, z_2 \in \BbC$.

Let $q > 0$ be minimal such that
    $\kappa^q \in \mathbb{R}$.
Let \[ M' = \begin{pmatrix}
\kappa^q & 1       \\
  0     & \kappa^q
\end{pmatrix} \]
and consider the set $F = \{(a_0 a_1 a_2 a_3 a_4 \dots)\}$ where
   \[ a_j = \left\{
    \begin{array}{ll}
    -1 & \mathrm{if}\ \Im(\ka^j) < 0 \\
    +1 & \mathrm{if}\ \Im(\ka^j) > 0 \\
    -1\ \mathrm{or}\ +1 & \mathrm{if}\ \Im(\ka^j) = 0
    \end{array} \right.
\ \ \mathrm{for\ all}\ j\]
and let $\mathcal{F} = \pi_M(F)$.
We note that $\Im(\ka^j) = 0$ if and only if $q \mid j$.

Let $s = \max(\Im(z_2): (z_1, z_2) \in A_{M})$.
We see that $(z_1, z_2) \in \mathcal{F}$ if and only if $\Im(z_2) = s$.
Furthermore, we see that there is a map $\varphi$ from $\{\pm 1\}^\BbN$ to $\mathcal{F}$ given by
    \[
    \varphi(b_0 b_1 b_{2} \dots) =  \pi_M(b_0 a_1 a_2 \dots a_{q-1} b_1 a_{q+1} \dots),
    \]
    where $a_1, a_2, a_3,$ etc are chosen to as above. The map~$\varphi$ is one-to-one,
    and moreover, it is clearly H\"older continuous in the standard metric.
This gives us that if a point is unique in $A_{M'}$, then the
    corresponding point in $\mathcal{F}$ is unique, from which the result follows.
\end{proof}

\section{Complex eigenvalues}
\label{sec:simple block}

\begin{lemma}\label{lem:two-rational}
Let $\ka_1, \ka_2, \dots, \ka_d$ be such that
    $\arg(\ka_j)/\pi \in \BbQ$.
Let $q > 0$ be minimal such that $\ka_j^q \in \BbR$ for all $j$.
\[
M=\begin{pmatrix} \kappa_1 &   &  & 0       \\
                    & \kappa_2 &  &         \\
                    &          & \ddots &   \\
                  0 &          & & \kappa_d
\end{pmatrix},
M'=\begin{pmatrix} \kappa_1^q &   &  & 0       \\
                    & \kappa_2^q &  &         \\
                    &          & \ddots &   \\
                  0 &          & & \kappa_d^q
\end{pmatrix}.
\]
We have $\dim_H(\mathcal U_{M'}) > 0$ if and only if
   $\dim_H(\mathcal U_{M}) > 0$.
\end{lemma}

\begin{proof}
Similarly to the proof of the complex part of Lemma~\ref{lem:jordan},
consider the set $F = \{(a_0 a_1 a_2 a_3 a_4 \dots)\}$ where
   \[ a_j = \left\{
    \begin{array}{ll}
    -1 & \mathrm{if}\ \Im(\ka_1^j) < 0 \\
    +1 & \mathrm{if}\ \Im(\ka_1^j) > 0 \\
    -1 & \mathrm{if}\ \Im(\ka_1^j) = 0\ \mathrm{and}\ \Im(\ka_2^j) < 0 \\
    +1 & \mathrm{if}\ \Im(\ka_1^j) = 0\ \mathrm{and}\ \Im(\ka_2^j) > 0 \\
    \vdots &  \vdots \\
    -1 & \mathrm{if}\ \Im(\ka_1^j) = \Im(\ka_2^j) = \dots = \Im(\ka_n^j) = 0\ \mathrm{and}\ \Im(\ka_{n+1}^j) < 0 \\
    +1 & \mathrm{if}\ \Im(\ka_1^j) = \Im(\ka_2^j) = \dots = \Im(\ka_n^j) = 0\ \mathrm{and}\ \Im(\ka_{n+1}^j) > 0 \\
    \vdots &  \vdots \\
    -1\ \mathrm{or}\ +1 & \mathrm{if}\ \Im(\ka_1^j) = \Im(\ka_2^j) = \dots = \Im(\ka_d^j) = 0
    \end{array} \right. \]
for all $j$ and let $\mathcal{F} = \pi(F)$.
We note that $\Im(\ka_1^j) =  \cdots = \Im(\ka_d^j) = 0$ if and only if $q \mid j$.

Put
\begin{align*}
s_1 & = \max(\Im(z_1): (z_1, \dots, z_d) \in A), \\
s_2 & = \max(\Im(z_2): (z_1, \dots, z_d) \in A, \Im(z_1) = s_1), \\
s_3 & = \max(\Im(z_3): (z_1, \dots, z_d) \in A, \Im(z_1) = s_1, \Im(z_2) = s_2), \\
    & \hspace{1cm} \vdots \\
s_d & = \max(\Im(z_d): (z_1, \dots, z_d) \in A, \Im(z_1) = s_1,
             \Im(z_2) = s_2, \dots,  \Im(z_{d-1}) = s_{d-1}).
\end{align*}
We see that $(z_1, z_2, \dots, z_d) \in \mathcal{F}$ if and only if $\Im(z_j) = s_j$
    for $j = 1, 2, \dots, d$.
Furthermore, the map $\psi: \{\pm 1\}^\BbN \to \mathcal{F}$ defined by
    \[
   \psi(b_0 b_1 \dots) =  \pi_M(b_0 a_1 a_2 \dots a_{q-1} b_1 a_{k+1} a_{k+2} \dots) ,
    \]
where the $a_1, a_2, a_3,$ etc are chosen as above, is one-to-one and H\"older continuous.
This gives us that if a point is unique in $A_{M'}$, then the
    corresponding point in $\mathcal{F}$ is unique. Moreover, $\dim_H \U_{M'} > 0$ implies
     $\dim_H \U_{M} > 0$.

For the other direction, assume that that
    $x = \pi_M(a_0 a_1 a_2 \dots)$ is in $\U_M$.
Consider the point $\pi_{M'}(a_0 a_q a_{2q} \dots)\in A_{M'}$.
If it is not a point of uniqueness, then there exists a
    $\pi_{M'}(b_0 b_q b_{2q} \dots) = \pi_{M'}(a_0 a_q a_{2q} \dots)$.
But by construction
    $x = \pi_M(b_0 a_1 a_2 \dots a_{q-1} b_q a_{q+1} \dots)$, a contradiction.

A similar argument can be used for the subsequence
    $a_{j} a_{q+j} a_{2 q + j} \dots$ mapping to a
    simple linear transformation of $\U_{M'}$, namely, $M^j \U_{M'}$.

Hence for any point of uniqueness in $\U_M$ we have $q$ maps into affine
    copies of $\U_{M'}$, each one giving a point of uniqueness.
If $\dim_H \U_M > 0$, then one of these maps will also have have positive
    Hausdorff dimension, from which the result follows.
\end{proof}

Now we are ready to conclude the proof of Theorem~\ref{thm:new}. Note first
that if $|\ka_1| \neq |\ka_2|$, then $|\ka_1^q| \neq |\ka_2^q|$ with
    $\ka_1^q, \ka_2^q \in \BbR$.
From this Theorem~\ref{thm:new}~(\ref{case:different}) follows from
Theorem~\ref{thm:old}~(\ref{case2}) or (\ref{case3}).

If $|\ka_1| = \dots = |\ka_d|=\lambda$, then $M'=\lambda^q J$, where $J$ is a
$d\times d$ diagonal matrix with $-1$ or $1$ on the diagonal.
If there exists $j$ and $k$ such that $\ka_j^q \ka_k^q < 0$, then $J$ will
    contain both a $-1$ and a $1$, and this will follow from
    Theorem~\ref{thm:old}~(\ref{case:signs}).
If no such $j$ and $k$ exists, then the result follows from
    \cite[Theorem~2]{GS}.

\end{document}